\documentclass{amsart}
\usepackage{amssymb}
\usepackage{hyperref}
\newtheorem{theorem}{Theorem}[section]
\newtheorem{lemma}[theorem]{Lemma}
\newtheorem{corollary}[theorem]{Corollary}
\theoremstyle{definition}
\newtheorem{remark}[theorem]{Remark}

\def\Aut{{\rm Aut}}

\newcommand{\ZZ}{\mathbb{Z}}

\begin{document}
\title[Vertex-primitive digraphs]{Vertex-primitive digraphs having vertices with almost equal neighbourhoods}

\author[P. Spiga, G. Verret]{Pablo Spiga, Gabriel Verret}

\address{Pablo Spiga,  Dipartimento di Matematica Pura e Applicata, 
\newline\indent University of Milano-Bicocca, Milano, 20125 Via Cozzi 55, Italy.} 
\email{pablo.spiga@unimib.it}

\address{Gabriel Verret,  School of Mathematics and Statistics, 
\newline\indent University of Western Australia, 35 Stirling Highway, Crawley, WA 6009, Australia.
\newline\indent FAMNIT, University of Primorska, Glagolja\v{s}ka 8, SI-6000 Koper, Slovenia.}
\email{gabriel.verret@uwa.edu.au}

\thanks{The last author is supported by UWA as part of the Australian Research Council grant DE130101001.}

\subjclass[2010]{Primary 05E18; Secondary 20B25,20Mxx}
\keywords{vertex-primitive digraphs, synchronising groups}

\begin{abstract}
We consider vertex-primitive digraphs having two vertices with almost equal neighbourhoods (that is, the set of vertices that are neighbours of one but not the other is small). We prove a structural result about such digraphs and then apply it to answer a question of Ara\'{u}jo and Cameron about synchronising groups.
\end{abstract}

\maketitle

\section{Introduction}
All sets, digraphs and groups considered in this paper are finite. For basic definitions, see Section~\ref{sec:prelim}. Let $\Gamma$ be a digraph on a set $\Omega$  and suppose that $\Gamma$ is \emph{vertex-primitive}, that is, its automorphism group acts primitively on $\Omega$. It is well known and easy to see that, in this case, $\Gamma$ cannot have two distinct vertices with equal neighbourhoods, unless $\Gamma=\emptyset$ or $\Gamma=\Omega\times\Omega$ (see Lemma~\ref{newLemma} for example). 

We consider the situation when $\Gamma$ has two vertices with ``almost'' equal neighbourhoods. Since $\Gamma$ is vertex-primitive, it is regular, of valency $d$, say. Let $\Gamma_i$ be the graph on $\Omega$ with two vertices being adjacent if the intersection of their neighbourhoods in $\Gamma$ has size $d-i$. Our main result is the following.

\begin{theorem}\label{theorem:main}
Let $\Gamma$ be a vertex-primitive digraph on a set $\Omega$ with $\Gamma\neq\emptyset$ and $\Gamma\neq\Omega\times\Omega$. Let $n$ be the order of $\Gamma$ and $d$ its valency. If $\kappa$ is the smallest positive $i$ such that $\Gamma_i\neq\emptyset$, then either 
\begin{enumerate}
\item $\Gamma_0\cup\Gamma_\kappa=\Omega\times\Omega$ and  $(n-1)(d-\kappa)=d(d-1)$, or \label{partfirst}
\item there exists $i\in\{\kappa,\ldots,d-1\}$ such that $\Gamma_i$ has valency at least $1$ and at most $\kappa^2+\kappa$. \label{parttwo}
\end{enumerate}
\end{theorem}

Theorem~\ref{theorem:main} is most powerful when $\kappa$ is small. To illustrate this, we completely determine the digraphs which occur when $\kappa=1$ (see Corollary~\ref{usefulCor}). We then apply this result to answer a question of Ara\'{u}jo and Cameron~\cite[Problem 2(a)]{ArCam} (see Theorem~\ref{theotheo}) concerning synchronising groups. Finally, in Section~\ref{Sec:othercase}, we say a few words about the case $\Gamma_0\cup\Gamma_\kappa=\Omega\times\Omega$.

\section{Preliminaries}\label{sec:prelim}

\subsection{Digraphs}
Let $\Omega$ be a finite set. A \emph{digraph} $\Gamma$ on $\Omega$ is a binary relation on $\Omega$, in other words, a subset of $\Omega\times\Omega$. The elements of $\Omega$ are called the \emph{vertices} of $\Gamma$ while the cardinality of $\Omega$ is the \emph{order} of $\Gamma$. The digraph $\Gamma^{-1}$ is $\{(\alpha,\beta)\in \Omega\times \Omega: (\beta,\alpha)\in\Gamma\}$. Given two digraphs $\Gamma$ and $\Lambda$ on $\Omega$, we define the digraph
 $$\Gamma\circ\Lambda:=\{(\alpha,\beta)\in \Omega\times \Omega: \textrm{there exists }\gamma\in \Omega \textrm{ with }(\alpha,\gamma)\in \Gamma,(\gamma,\beta)\in \Lambda\}.$$
 
Let $v$ be a vertex of $\Gamma$. The \emph{neighbourhood} of $v$ is the set $\{u\in\Omega:(v,u)\in\Gamma\}$ and is denoted $\Gamma(v)$. Its cardinality is the \emph{valency} of $v$. If every vertex of $\Gamma$ has the same valency, say $d$, then we say that $\Gamma$ is \emph{regular of valency $d$}.

If $\Gamma$ is a symmetric binary relation, then it is sometimes called a \emph{graph}. If $\Psi$ is a subset of $\Omega$, then the \emph{subgraph of $\Gamma$ induced by $\Psi$} is $\Gamma\cap(\Psi\times\Psi)$ viewed as a graph on $\Psi$. We denote by $\Omega^*$ the set $\{(v,v):v\in\Omega\}$. The graph $(\Omega\times\Omega)\setminus\Omega^*$ is called the \emph{complete graph} on $\Omega$.

\subsection{Groups}
The \emph{automorphism group} of $\Gamma$, denoted $\Aut(\Gamma)$, is the group of permutations of $\Omega$ that preserve $\Gamma$. A permutation group $G$ on $\Omega$ is \emph{transitive} if for every $x,y\in\Omega$ there exists $g\in G$ with $x^g=y$, that is, mapping $x$ to $y$. A permutation group on $\Omega$ is \emph{primitive} if it preserves no nontrivial partition of $\Omega$. (Observe that a primitive group $G$ on $\Omega$ is transitive unless $G=1$ and $|\Omega|=2$.) We say that $\Gamma$ is \emph{vertex-transitive}  if $\Aut(\Gamma)$ is transitive.



\subsection{A few basic results}
\begin{lemma}\label{super basic}
Let $\Gamma$ be a vertex-transitive digraph on $\Omega$. If $\Gamma_0=\Omega\times\Omega$, then $\Gamma=\emptyset$ or $\Gamma=\Omega\times\Omega$.
\end{lemma}
\begin{proof}
Suppose that $\Gamma\neq\emptyset$ and thus there exists $(\alpha,\beta)\in\Gamma$. As $\Gamma_0=\Omega\times\Omega$, all vertices of $\Gamma$ have the same neighbourhood and thus $\beta\in\Gamma(\omega)$ for every $\omega\in\Omega$ but then vertex-transitivity implies that $\Gamma=\Omega\times\Omega$.
\end{proof}

\begin{lemma}\label{super basic2}
Let $\Gamma$ be a vertex-primitive graph on $\Omega$. If $\Gamma$ is a transitive relation on $\Omega$ and $\Gamma\not\subseteq\Omega^*$, then $\Gamma=\Omega\times\Omega$.
\end{lemma}

\begin{lemma}\label{newLemma}
Let $\Gamma$ be a vertex-primitive digraph on $\Omega$. If $\emptyset\neq\Gamma\neq\Omega\times\Omega$, then $\Gamma_0=\Omega^*$.
\end{lemma}
\begin{proof}
Clearly, $\Omega^*\subseteq \Gamma_0$. If $\Gamma_0\not\subseteq\Omega^*$, then Lemma~\ref{super basic2} implies that $\Gamma_0=\Omega\times\Omega$ but this contradicts Lemma~\ref{super basic}.
\end{proof}

\section{Proof of Theorem~\ref{theorem:main}}

Since $\emptyset\neq\Gamma\neq\Omega\times\Omega$, $n\geq 2$ and Lemma~\ref{super basic} implies that $\Gamma_0\neq \Omega\times\Omega$ and thus $\kappa$ is well-defined. Since $\Gamma$ is vertex-primitive, it is regular, of valency $d$, say. By minimality of $\kappa$, $\Gamma_i=\emptyset$ for $i\in\{1,\ldots,\kappa-1\}$ but $\Gamma_\kappa\neq\emptyset$. Note that, for every integer $i$,  $\Aut(\Gamma)\leq\Aut(\Gamma_i)$ hence $\Gamma_i$ is also vertex-primitive and regular, of valency $d_i$, say. By Lemma~\ref{newLemma}, we have $\Gamma_0=\Omega^*$ and thus 
\begin{equation}\label{eq1}
d_0=1.
\end{equation}
Moreover, it is easy to see that 
\begin{equation}\label{EqEq}
\bigcup_{i\in\{0,\ldots,d-1\}}\Gamma_i=\Gamma\circ\Gamma^{-1}.
\end{equation}
It is also easy to check that, for every integer $i$, we have
\begin{equation}\label{Gammaj1}
\Gamma_i\circ \Gamma_\kappa\subseteq \Gamma_{i-\kappa}\cup\Gamma_{i-\kappa+1}\cup\cdots\cup\Gamma_{i+\kappa}.
\end{equation}
Let 
\begin{equation}\label{Defell}
\ell:=\min\{i\geq\kappa : \Gamma_{i+1}=\Gamma_{i+2}=\cdots=\Gamma_{i+\kappa}=\emptyset\}.
\end{equation}
By definition, we have $\ell\geq \kappa$. Recall that $\Gamma_i=\emptyset$ for every $i\geq d+1$ hence
\begin{equation}\label{ellbound}
\kappa\leq\ell\leq d.
\end{equation}
Let $\overline{\Gamma_\kappa}$ be the transitive closure of the relation $\Gamma_\kappa$. (That is, the minimal transitive relation containing $\Gamma_\kappa$.) By Lemma~\ref{super basic2}, we have $\overline{\Gamma_\kappa}=\Omega\times\Omega$. 

Let 
$$\Lambda:=\Gamma_0\cup\Gamma_1\cup \cdots\cup \Gamma_{\ell}.$$
Note that
\begin{eqnarray*}
\Lambda\circ \Gamma_\kappa&=&(\Gamma_0\circ \Gamma_\kappa)\cup(\Gamma_1\circ \Gamma_\kappa)\cup \cdots \cup (\Gamma_{\ell}\circ \Gamma_\kappa)\\
&{(\ref{Gammaj1})}\atop{\subseteq}& \Gamma_{-\kappa}\cup\Gamma_{-\kappa+1}\cup \cdots\cup \Gamma_{\kappa+\ell}\\
&{(\ref{Defell})}\atop{=}& \Gamma_0\cup\Gamma_1\cup \cdots\cup \Gamma_{\ell}=\Lambda.
\end{eqnarray*}
As $\Gamma_\kappa\subseteq\Lambda$, it follows by induction that $\overline{\Gamma_\kappa}\subseteq \Lambda$ and thus 
$$\Lambda=\Omega\times\Omega.$$
This implies that
\begin{equation}\label{EqEq2}
n=\sum_{i=0}^{\ell} d_i=1+d_1+\cdots +d_{\ell}.
\end{equation}
We now consider two cases, according to whether $\ell=\kappa$ or $\ell\geq\kappa+1$.

\subsection{$\boldsymbol{\ell=\kappa}$}\label{papapa}

If $\ell=\kappa$, then the minimality of $\kappa$ implies $\Lambda=\Gamma_0\cup\Gamma_\kappa$ and hence $\Gamma_0\cup\Gamma_\kappa=\Omega\times\Omega$. Let $\mathcal{B}:=\{\Gamma(\alpha)\mid \alpha\in \Omega\}$ and let 
$$\mathcal{S}:=\{(\alpha,b,b')\mid \alpha\in \Omega,b,b'\in\mathcal{B},\alpha\in b\cap b',b\neq b'\}.$$ 
Now, 
$$|\mathcal{S}|=\sum_{\alpha\in \Omega}|\{(b,b')\mid \alpha\in b\cap b',b\neq b'\}|=\sum_{\alpha\in \Omega}d(d-1)=nd(d-1).$$ 
On the other hand, observe that $|b\cap b'|=d-\kappa$ for every $b,b'\in\mathcal{B}$ with $b\neq b'$ and thus
$$|\mathcal{S}|=\sum_{\substack{b,b'\in\mathcal{B}\\b\neq b'}}|b\cap b'|=\sum_{\substack{b,b'\in\mathcal{B}\\b\neq b'}}(d-\kappa)=n(n-1)(d-\kappa).$$
 Therefore $(n-1)(d-\kappa)=d(d-1)$ and the theorem follows.

\subsection{$\boldsymbol{\ell\geq\kappa+1}$}
We assume that $\ell\geq\kappa+1$. By minimality of $\ell$, we have $\Gamma_{\ell}\neq\emptyset$ and thus  there exist two vertices of $\Gamma$ whose neighbourhoods intersect in $d-\ell$ vertices hence, considering the union of their neighbourhoods, we obtain 
\begin{equation}\label{EqEq3}
n\geq d+\ell.
\end{equation}
Let $\alpha\in\Omega$ and let
$$\mathcal{S}(\alpha):=\{(\beta,\gamma)\in\Omega\times \Omega: \beta\in \Gamma(\alpha)\cap\Gamma(\gamma)\}.$$
Clearly, 
\begin{equation}\label{eq:basic}
|\mathcal{S}(\alpha)|=\sum_{\beta\in\Gamma(\alpha)}|\{\gamma\in \Omega: \gamma\in \Gamma^{-1}(\beta)\}|=\sum_{\beta\in\Gamma(\alpha)}d=d^2.
\end{equation}
On the other hand,
\begin{eqnarray*}
|\mathcal{S}(\alpha)|&=&\sum_{\gamma\in (\Gamma\circ \Gamma^{-1})(\alpha)}|\Gamma(\alpha)\cap \Gamma(\gamma)|\\
&{(\ref{EqEq})}\atop{=}&\sum_{i=0}^{d-1}\sum_{\gamma\in \Gamma_i(\alpha)}|\Gamma(\alpha)\cap\Gamma(\gamma)|=\sum_{i=0}^{d-1}\sum_{\gamma\in \Gamma_i(\alpha)}(d-i)=\sum_{i=0}^{d}d_i(d-i)\\
&{(\ref{ellbound})}\atop{\geq}&\sum_{i=0}^{\ell}d_i(d-i)\\
&{(\ref{eq1})}\atop{=}&d+\sum_{i=1}^{\ell}d_i(d-i)\\
&{(\ref{EqEq2})}\atop{=}&d+\left(n-1-\sum_{i=1}^{\ell-1}d_i\right)(d-\ell)+\sum_{i=1}^{\ell-1}d_i(d-i)\\
&=&d+(n-1)(d-\ell)+\sum_{i=1}^{\ell-1}d_i(\ell-i)\\
&{(\ref{EqEq3})}\atop{\geq}& d^2-\ell(\ell-1)+\sum_{i=1}^{\ell-1}d_i(\ell-i).\\
\end{eqnarray*}
Combining this with (\ref{eq:basic}), we have
\begin{equation*}
\ell(\ell-1)\geq\sum_{i=1}^{\ell-1}d_i(\ell-i).
\end{equation*}
Let $\mathcal{I}:=\{i\in\{1,\ldots,\ell-1\}:d_i\neq 0\}$ and $d^*:=\min\{d_i:i\in\mathcal{I} \}$. We have 
\begin{equation}\label{eq:mix}
\ell(\ell-1)\geq d^*\sum_{i\in\mathcal{I}}(\ell-i).
\end{equation}
Since $\ell\geq\kappa+1$, $\kappa$ is the minimum element of $\mathcal{I}$. Note also that $(\ell-i)$ is decreasing with respect to $i$ and, by definition of $\ell$, any two elements of $I$ are at most $\kappa$ apart. Let 
$$\sigma:=\left\lfloor\frac{\ell-1}{\kappa}\right\rfloor.$$
Then
\begin{eqnarray*}
\sum_{i\in\mathcal{I}}(\ell-i)&\geq& (\ell-\kappa)+(\ell-2\kappa)+\cdots+(\ell-\sigma\kappa)\\
&=& \sigma\ell -\frac{\kappa\sigma(\sigma+1)}{2}.
\end{eqnarray*}
Combining this with (\ref{eq:mix}), we find
\begin{equation}\label{yoyoyo}
\ell(\ell-1)\geq d^*\left(\sigma\ell -\frac{\kappa\sigma(\sigma+1)}{2}\right).
\end{equation}
Write $\ell:=\sigma\kappa+r$, with $r\in \{1,\ldots,\kappa\}$. Now, \eqref{yoyoyo} gives
$$\frac{2(\sigma\kappa+r)(\sigma\kappa+r-1)}{\sigma(\sigma\kappa+2r-\kappa)}\geq d^*.$$
Calculating the derivative of the left-hand-side with respect to $\sigma$, one finds
$$-\frac{2\left(\sigma^2\kappa^2(\kappa-1)+ r(r-1)((2\sigma-1)\kappa+2r)   \right)}{(\sigma(\sigma\kappa+2r-\kappa))^2},$$
which is clearly nonpositive since $r,\kappa,\sigma\geq 1$. It follows that the maximum of the left-hand-side of \eqref{yoyoyo} is attained when $\sigma=1$, hence
\begin{equation}\label{yoyoyoyo}
\frac{(\kappa+r)(\kappa+r-1)}{r}\geq d^*.
\end{equation}
If $r=\kappa$, then the left-hand-side of \eqref{yoyoyoyo} is $4\kappa-2$ and an easy computation shows that $4\kappa-2\leq \kappa^2+\kappa$. If $r\leq\kappa-1$, then another easy computation yields that the left-hand-side of \eqref{yoyoyoyo} is a decreasing function of $r$, hence the minimum is attained when $r=1$ and $\kappa^2+\kappa\geq d^*$.
%
%
This completes the proof.\hfill \qed

\begin{remark}
The upper bound $\kappa^2+\kappa$ in Theorem~\ref{theorem:main}~(\ref{parttwo}) is actually tight (see some of the examples in Section~\ref{sec1}). However, the proof of Theorem~\ref{theorem:main} reveals that when more information about $\Gamma$ is available, this upper bound can be drastically improved. For instance, following the argument in the last part of the proof, one finds that if $\sigma\geq 2$,  then $d^*\leq (4\kappa^2+2\kappa)/(\kappa+2)\leq 4\kappa-2$ and hence there exists $i\in\{\kappa,\ldots,d-1\}$ such that $\Gamma_i$ has nonzero valency bounded by a linear function of $\kappa$.
\end{remark}

\section{The case $\kappa=1$ and an application to synchronising groups}\label{sec1}
We now completely determine the digraphs that arise when $\kappa=1$ in Theorem~\ref{theorem:main}. Let $p$ be a prime, let $d\in \ZZ$ with $0\leq d\leq p$ and let $x\in\ZZ_p$. We define $\Delta_{p,x,d}$ to be the Cayley digraph on $\ZZ_p$ with connection set $\{x+1,x+2,\ldots,x+d\}$. (That is, $(u,v)\in\Delta_{p,x,d}$ if and only if $v-u\in\{x+1,x+2,\ldots,x+d\}$.) We will need the following easy lemma.

\begin{lemma}\label{lemma:d1}
Let $\Gamma$ be a vertex-primitive digraph on $\Omega$ of valency $1$. Then $\Gamma=\Omega^*$ or $\Gamma\cong\Delta_{p,0,1}$ for some prime $p$.
\end{lemma}
\begin{proof}
If $\Gamma\cap\Omega^*\neq\emptyset$, then, by vertex-primitivity, $\Gamma=\Omega^*$. Otherwise, $\Gamma$ must be a directed cycle and, again by vertex-primitivity, must have prime order $p$ and hence $\Gamma\cong \Delta_{p,0,1}$.
\end{proof}

\begin{corollary}\label{usefulCor}
Let $\Gamma$ be a vertex-primitive digraph on $\Omega$. If $\Gamma_1\neq\emptyset$, then one of the following occurs:
\begin{enumerate}
\item $\Gamma=\Omega^*$, 
\item $\Gamma$ is a complete graph, or 
\item $\Gamma\cong\Delta_{p,x,d}$, for some prime $p$ and $d\geq 1$.
\end{enumerate}
\end{corollary}
\begin{proof}
Let $n$ be the order of $\Gamma$ and $d$ its valency. Since $\Gamma_1\neq\emptyset$, $d\geq 1$. If $d=1$, then the result follows by Lemma~\ref{lemma:d1}. We thus assume that $d\geq 2$. By Theorem~\ref{theorem:main}, either $n=d+1$, or there exists $i\in\{1,\ldots,d-1\}$ such that $\Gamma_i$ is regular of valency at least $1$ and at most $2$. 

Suppose first that  $n=d+1$. This implies that $(\Omega\times\Omega)\setminus\Gamma$ has valency $1$ and the result follows again by Lemma~\ref{lemma:d1}.

We may thus assume that there exists $i\in\{1,\ldots,d-1\}$ such that $\Gamma_i$ is regular of valency at least $1$ and at most $2$. This implies that $\Gamma_i$ must have order $2$ or  be a vertex-primitive cycle and thus have prime order. It follows that $\Gamma$ also has prime order and is thus a Cayley digraph on $\ZZ_p$ for some prime $p$. Up to isomorphism, we may assume that $(0,1)\in\Gamma_1$. Let $y$ be the unique element of $\Gamma(1)\setminus\Gamma(0)$. Now, for every $s\in\Gamma(0)\setminus\{y-1\}$, we have that $s+1\in\Gamma(1)\setminus\{y\}$ and thus $s+1\in\Gamma(0)$. It follows that $\Gamma(0)$ is of the form $\{x+1,x+2,\ldots,y-1\}$ for some $x$ and the result follows.
\end{proof}

We note that the second author asked for a proof of Corollary~\ref{usefulCor} on the popular MathOverflow website (see~\url{http://mathoverflow.net/q/186682/}). The question generated some interest there but no answer.

We now use Corollary~\ref{usefulCor} to answer Problem 2(a) in~\cite{ArCam}. In fact, we will prove a slightly more general result. First, we need some terminology regarding synchronising groups (see also~\cite{CameronCourse}). Let $G$ be a permutation group and let $f$ be a map, both with domain $\Omega$. The \emph{kernel} of $f$ is the partition of $\Omega$ into the inverse images of points in the image of $f$. The \emph{kernel type} of $f$ is the partition of $|\Omega|$ given by the sizes of the parts of its kernel. We say that $G$ \emph{synchronises} $f$ if the semigroup $\langle G,f\rangle$ contains a constant map, while $G$ is said to be \emph{synchronising} if $G$ synchronises every non-invertible map on $\Omega$.

\begin{theorem}\label{theotheo}
Let $\Omega$ be a set, let $G$ be a primitive permutation group on $\Omega$ and let $f$ be a map on $\Omega$. If $f$ has kernel type $(p,2,1,\ldots,1)$ with $p\geq 2$, then $G$ synchronises $f$.
\end{theorem}
\begin{proof}
By contradiction, we assume that $G$ does not synchronise $f$. Let $\Gamma$ be the graph on $\Omega$ such that $(v,w)\in\Gamma$ if and only if there is no element of $\langle G,f\rangle$ which maps $v$ and $w$ to the same point. 

By~\cite[Theorem~5(a),(b)]{ArCam}, we have $\Gamma\neq\emptyset$ and $G\leq\Aut(\Gamma)$ and thus $\Gamma$ is vertex-primitive. Since $f$ is not a permutation, $\Gamma$ is not complete.  Transitive groups of prime degree are synchronising (see for example~\cite[Corollary 2.3]{Neumann}) hence we may assume that $|\Omega|$ is not prime. It thus follows  by Corollary~\ref{usefulCor} that $\Gamma_1=\emptyset$. 

Let $d$ be the valency of $\Gamma$ and let $A$ and $B$ be the parts of the kernel of $f$ with sizes $2$ and $p$, respectively. Let $a=f(A)$ and $b=f(B)$, let $K=A\cup B$ and let $Y$ be the subgraph of $\Gamma$ induced by $K$. By definition, $Y$ is bipartite, with parts $A$ of size $2$ and $B$ of size $p$. By \cite[Lemma 10]{ArCam}, every vertex of $Y$ has degree at least one.

Suppose that there exist $b_1,b_2\in B$ having valency one in $Y$. Then $\Gamma(b_1)\setminus A$ and $\Gamma(b_2)\setminus A$ are mapped injectively and hence bijectively into $\Gamma(b)\setminus\{a\}$ hence we have $\Gamma(b_1)\setminus A=\Gamma(b_2)\setminus A$. This implies that $(b_1,b_2)\in\Gamma_1$, a contradiction.

Now suppose that there exist $b_1,b_2\in B$ having valency two in $Y$. Then $\Gamma(b_1)\cap K=A=\Gamma(b_2)\cap K$ and, as before, $\Gamma(b_1)\setminus A$ and $\Gamma(b_2)\setminus A$ are mapped injectively into $\Gamma(b)\setminus\{a\}$. Since $|\Gamma(b_1)\setminus A|=d-2=|\Gamma(b_2)\setminus A|$ while $|\Gamma(b)\setminus\{a\}|=d-1$, it follows that $|(\Gamma(b_1)\setminus A)\cap(\Gamma(b_1)\setminus A)|\geq d-3$ and thus $|\Gamma(b_1)\cap \Gamma(b_2)|\geq d-1$, which is again a contradiction.

Since every vertex of $B$ has valency either one or two in $Y$, we conclude that $|B|\leq 2$ thus $p=2$ and the result follows by~\cite[Theorem~3(a),(b)]{ArCam}.
\end{proof}

\section{The case $\Gamma_0\cup\Gamma_\kappa=\Omega\times\Omega$}\label{Sec:othercase}

We now say a few words about part~\eqref{partfirst} of the conclusion of Theorem~\ref{theorem:main}, that is, when $\Gamma_0\cup\Gamma_\kappa=\Omega\times\Omega$ and 
\begin{equation}\label{eq:1}
(n-1)(d-\kappa)=d(d-1).
\end{equation}
Let $\mathcal{B}:=\{\Gamma(\alpha)\mid \alpha\in \Omega\}$, as in Section~\ref{papapa}. Note that $\mathcal{B}$ is a set of $d$-subsets of $\Omega$. Moreover, any two distinct elements of $\mathcal{B}$ intersect in $d-\kappa$ elements. In particular, $(\Omega,\mathcal{B})$ is a symmetric  $2$-design with parameters $(n,d,d-\kappa)$ and with a point-primitive automorphism group. (For undefined terminology, see for example~\cite[Chapter~$1$]{CameronVanLint}.)

Given a specific value of $\kappa$, one can often push the analysis further and determine all the possibilities for $\Gamma$. Recall that $1\leq\kappa\leq d$. If $\kappa=d$, then, since $d\geq 1$, Eq. \eqref{eq:1} implies that $d=1$ and we may apply Lemma~\ref{lemma:d1}. The case when $\kappa=1$ was dealt with in Corollary~\ref{usefulCor}. From now on, we assume that 
$$2\leq\kappa\leq d-1.$$ 
Observe that now Eq. \eqref{eq:1} yields
\begin{equation}\label{eq:2}
n=d+\kappa+\frac{\kappa(\kappa-1)}{d-\kappa}.
\end{equation}
In particular, $d-\kappa\leq \kappa(\kappa-1)$, that is,
\begin{equation}\label{eq:3}
d\leq \kappa^2.
\end{equation}
A computation using Eq. \eqref{eq:2} also yields that, for fixed $\kappa$, $n$ is a non-decreasing function of $d$. Therefore the maximum for $n$ (as a function of $\kappa$) is achieved when $d=\kappa+1$ and $n\leq \kappa^2+\kappa+1$.

In our opinion, the most interesting situation occurs when $d=\kappa^2$ or (dually) when $d=\kappa+1$. By Eq. \eqref{eq:1}, we have $n=\kappa^2+\kappa+1$ and thus $(\Omega,\mathcal{B})$ is a symmetric  $2$-design with parameters $(\kappa^2+\kappa+1,\kappa+1,1)$, that is, a finite projective plane of order $\kappa$. Note that $\Aut(\Gamma)$ cannot be $2$-transitive and thus this is a non-Desarguesian projective plane. By a remarkable theorem of Kantor \cite[Theorem B (ii)]{Kan} (which depends upon the classification of the finite simple groups), $n$ is prime.

We conclude by showing how, given an explicit value of $\kappa$, one can often pin down the structure of $\Gamma$. We do this using $\kappa:=4$ as an example. By Eq.~\eqref{eq:2}, we have $n=d+4+12/(d-4)$ and hence $d\in \{5,6,7,8,10,16\}$. Moreover, replacing $\Gamma$ by its complement $(\Omega\times\Omega)\setminus \Gamma$ we may assume that $2d\leq n$. Therefore $(d,n)\in \{(5,21),(6,16),(7,15)\}$. The previous paragraph shows that the case $(d,n)=(5,21)$ does not arise because $21$ is not a prime. When $(d,n)=(7,15)$, a careful analysis of the primitive groups of degree $15$ reveals that $\Gamma$ is isomorphic to the Kneser graph with parameters $(6,2)$ with a loop attached at each vertex. Finally, if $(d,n)=(6,16)$, then going through the primitive groups of degree $16$, one finds that $\Gamma$ is isomorphic to either the Clebsch graph with a loop added at each vertex or to the cartesian product of two copies of a complete graph of order $4$.

\end{document}